\documentclass{amsart}
\usepackage{amsmath, amsthm, amssymb, amsfonts}
\usepackage{bm}
\usepackage{dsfont}
\usepackage[utf8]{inputenc}
\usepackage{rotate}
\usepackage{tikz}
\usepackage{tikz-cd}
\usepackage[arrow,matrix,curve]{xy} 	
\usepackage{xcolor}
\usepackage{lipsum}
\usepackage{caption}
\usepackage{subcaption}
\usepackage{todonotes}

\usepackage{verbatim}

\theoremstyle{plain}
\newtheorem{theorem}{Theorem}[section]
\newtheorem{proposition}[theorem]{Proposition}
\newtheorem{lemma}[theorem]{Lemma}

\newtheorem{corollary}[theorem]{Corollary}

\theoremstyle{definition}

\newtheorem{remark}[theorem]{Remark}
\newtheorem{example}[theorem]{Example}

\newcommand{\norm}[1]{\left\lVert#1\right\rVert}

\definecolor{darkpastelgreen}{rgb}{0.01, 0.75, 0.24}

\newcommand{\mathset}[2]{\left\{#1\middle\vert #2 \right\}}
\newcommand{\mathseq}[2]{\left(#1\right)_{#2}}
\newcommand{\Exp}[1]{\mathbb{E}\left(#1\right)}
\newcommand{\ExpCon}[2]{\mathbb{E}\left(#1\middle\vert #2\right)}
\newcommand{\seq}[2]{\left(#1\right)_{#2}}
\newcommand{\Prob}[1]{\mathbb{P}\left(#1\right)}
\newcommand{\ProbCon}[2]{\mathbb{P}\left(#1\middle\vert #2\right)}
\newcommand{\card}[1]{{\# #1}}
\newcommand{\ind}{{\mathds{1}}}
\newcommand{\N}{\mathbb{N}}
\newcommand{\Z}{\mathbb{Z}}
\newcommand{\R}{\mathbb{R}}

\newcommand{\leb}{\mathbb{\lambda}}

\newcommand*\diff{\mathop{}\!\mathrm{d}}

\DeclareMathOperator{\var}{Var}
\DeclareMathOperator{\uniform}{Unif}
\DeclareMathOperator{\BOX}{Box}

\title{Poissonian pair correlations for dependent random variables}
\date{\today}

\author{Jasmin Fiedler \and Michael Gnewuch \and Christian Weiß}
\address{Jasmin Fiedler, {\bf{Ruhr West University of Applied Sciences}}\\ {{Department of Natural Sciences, Duisburger Str. 100}}\\{{45479 M\"ulheim an der Ruhr, Germany}}, jasmin.fiedler@hs-ruhrwest.de\newline
Michael Gnewuch, {\bf{University of Osnabr\"uck}}\\ {{Institute of mathematics, Albrechtstr. 28a}}\\{{49076 Osnabr\"uck , Germany}}, michael.gnewuch@uni-osnabrueck.de\newline
Christian Wei\ss{}, {\bf{Ruhr West University of Applied Sciences}}\\ {{Department of Natural Sciences, Duisburger Str. 100}}\\{{45479 M\"ulheim an der Ruhr, Germany}}, christian.weiss@hs-ruhrwest.de
}
\email{}
\begin{document}

\begin{abstract} We consider Poissonian pair correlations (PPC) for uniformly distributed sequences of random numbers with a dependency structure. More specifically, we treat two classes of dependent random variables which have widely been studied in the literature, namely sequences of jittered samples and random walks on the torus. We show that for the former class, the PPC property depends on how the finite sample is extended to an infinite sequence. Moreover, we prove that, under some mild assumptions, the random walk on the torus generically has PPC.
\end{abstract}

\maketitle

\section{Introduction and main results}\label{sec:intro}

A random variable $X$ on the unit interval $[0,1)$ is said to be uniformly distributed, $X\sim\uniform[0,1)$, if its probability measure is the Lebesgue measure $\lambda$ restricted to $[0,1)$. It is a generic property of sequences $\mathseq{X_n}{n\in N}$ of independent, uniformly distributed random variables that
\[
\lim_{N\to\infty}\frac{1}{N}\card{\mathset{1\leq i\leq N}{X_i\in[a,b)}}=\leb([a,b))=b-a
\] for all $0\leq a<b\leq 1$. 
This property serves as the definition of uniform distribution for arbitrary (random or deterministic) sequences $x=\mathseq{x_n}{n\in \N} \subset [0,1]$. To avoid confusion we will be calling a random variable uniformly distributed and a deterministic sequence equidistributed. However the definition does not give any information on how quickly the limit approaches the Lebesgue measure. To quantify the speed of convergence, the notion of discrepancy was introduced and is nowadays a corner stone in equidistribution theory. It is defined by
    \[
    D_N(x):=\sup_{0\leq a<b\leq 1}\left\vert \frac{1}{N}\card{\mathset{1\leq i\leq N}{x_i\in[a,b)}}-\leb([a,b)) \right\vert.
    \]
A seminal result in \cite{Sch72} states that for any sequence $x$ the inequality
\[
D_N(x)\geq c \frac{\ln(N)}{N},
\]
holds for infinitely many $N\in\N$ and a universal constant $c>0$. Sequences with the optimal behavior $D_N(x)=\mathcal{O}\left(\frac{\ln(N)}{N},\right)$ are called low-discrepancy sequences, see also \cite{KN74}.\\[12pt]
While the discrepancy measures the uniformity of sequences on a global scale, the behavior and combinatorics of its gaps, i.e. the distances between its geometrically neighboring elements, is typically considered on a local scale: Given $0<\alpha\leq 1$, $N\in\N$ and $s\geq0$ we define
\begin{equation}\label{eq:def_R(s,N)}
 R_\alpha(s,N):=\frac{1}{N^{2-\alpha}} \card{\left\{ 1 \leq i \neq j \leq N : \norm{x_i-x_j} \leq \frac{s}{N^\alpha} \right\}},  
\end{equation}
where $\norm{\cdot}$ denotes the distance of a number to its closest integer. We say that $x=\mathseq{x_n}{n\in\N}$ has $\alpha$-Poissonian pair correlations ($\alpha$-PPC) if
\begin{equation}
    \lim_{N\to\infty} R_\alpha(s,N)=\leb([-s,s])=2s
\end{equation}
for all $s\geq 0$. If $\alpha=1$, we simply say that the sequence has Poissonian pair correlations (PPC), a notion popularized in \cite{rudnick:pair_corr_polynom}. For brevity, we will also write $R(s,N):= R_1(s,N)$. Possessing Poissonian pair correlations is another generic property of independent and uniformly distributed random sequences, see e.g. \cite{Mar07}. Besides these probabilistic, i.e. not explicit, ones only relatively few examples of sequences with Poissonian pair correlations have been found so far, see e.g. \cite{elbaz:sqrt_n_poiss,radz:alphantheta_poiss,lutsko:alphantheta_poiss}. Moreover, the case of mixtures of deterministic and stochastic sequences has been treated in the literature, see \cite{lachman:additive_energy_discrepancy_poissonian,SCHMIEDT2024422}. Nonetheless, it is known that any sequence of numbers in $[0,1)$ which has PPC is equidistributed according to \cite[Theorem 1]{aistleitner:pair_corr_equi}.\\[12pt] 
If $\alpha<1$, then the property is also called weak Poissonian pair correlations, see \cite{NP07}. This wording can be justified by the fact that, for $0<\alpha_1<\alpha_2\leq 1$ it is known that $\alpha_2$-Poissonian pair correlations imply $\alpha_1$-Poissonian pair correlations, see for example \cite[Theorem~4]{hauke:weak_poiss_corr}. It is interesting to note, that many classical low-discrepancy sequences such as van-der-Corput sequences and Kronecker sequences (see e.g. \cite{KN74}) do not have Poissonian pair correlations.\footnote{For these specific examples this is due to their gap structure, see \cite[Proposition 1]{larcher:som_neg_results_poiss_pair_corr}. Nonetheless, it is not known if low-discrepancy and PPC are mutually exclusive.} However, the situation is different for weak Poissonian pair correlations. In \cite{weiss:some_conn_discr_fin_gap_pair_corr} it was shown that any sequence $x=\mathseq{x_n}{n\in\N}$ in $[0,1)$ with $D_N(x)=o(N^{-(1-\varepsilon)})$ for $0<\varepsilon<1$ has $\alpha$-PPC as long as $0 < \alpha\leq 1 - \varepsilon$.
Therefore, low-discrepancy sequences have $\alpha$-PPC for all $\alpha<1$.\\[12pt]
While Poissonian pair correlations are a generic property of independent, uniformly distributed sequences of random variables one might ask what happens if one of these conditions is dropped. The case of independent, non uniformly distributed sequences was essentially solved in \cite{aistleitner:pair_corr_equi}. However to the best of the authors' knowledge, only limited research has been conducted for dependent random variables and it concentrated on the setting of certain subsequences of Kronecker sequences, see most importantly \cite{BTW01}. The present paper therefore aims to contribute to this question by looking at two typical classes of sequences of \textit{dependent} random variables, namely random walks on the torus and jittered sampling. We can give complete answers for the pair correlation statistic in these cases.\\[12pt]
First, we consider a random walk on $[0,1)$ or more precisely on the torus $\mathbb{T}_1=[0,1]\, /\sim$, where the ends of the unit interval are glued together. In more detail: Let $\mathseq{Y_n}{n\in\N}$ be independent and identically distributed (i.i.d.) random variables on $[0,1)$ with distribution $D$. We now define $\mathseq{X_n}{n\in\N}$ recursively by setting $X_1=x_1\in[0,1)$ and $X_{N+1}=\{X_N+Y_N\}$, where $\{ \cdot \}$ denotes the fractional part of a number. This is a one-dimensional random walk on the torus starting at $x_1$ whose step size has distribution $D$. As long as $D$ has an $L^p$-Lebesgue density for some $p>1$, the sequence $\mathseq{X_n}{n\in\N}$ possess Poissonian pair correlations.
\begin{theorem}\label{thm:rand_walk_ppc}
    Let $X=\mathseq{X_n}{n\in\N}$ be a random walk on the torus starting in ${x_1\in[0,1)}$ with distribution $D$. 
    Assume that $D$ has a Lebesgue density in $L^p([0,1))$ for some $p\in [1,\infty]$. 
    \begin{itemize}
        \item[(i)]
        If $p>1$, then $X$ almost surely has PPC.
        \item[(ii)]
        If $p=1$, then we still have
        $\lim_{N\to \infty} \Exp{R(s,N)} = 2s$ for all $s \geq 0$.
    \end{itemize} 
\end{theorem}
Indeed, our proof relies on the fact that the distribution $D$ has a Lebesgue density because it builds upon earlier results from \cite{schatte:asym_distr_sum} on random walks on the unit torus, cf. Theorem~\ref{thm:schatte_approx_unif}, 
and some additional probability estimates, cf. Lemma~\ref{lem:pre_randomsum}. This requirement cannot be dropped in general and it is easy to construct counterexamples.
\begin{example}
 If $x_1=0$ and every $Y_n$ satisfies $\Prob{Y_n=0}=\Prob{Y_n=\frac12}=\frac12$, then $X_n$ only takes the values $0$ and $\frac12$, which means that 
 the resulting random walk takes only values in $\{0, 1/2$\} 
 and therefore cannot have $\alpha$-PPC for any $\alpha\in(0,1]$.\\
 However, if $\Prob{Y_n=c}=1$ for some badly approximable $c \in \mathbb{R}$, then the sequence is a Kronecker sequence which has $\alpha$-PPC for every $\alpha<1$ but not for $\alpha =1$, see e.g. \cite{weiss:some_conn_discr_fin_gap_pair_corr}.
\end{example}
The second class of dependent sequences which we consider here is based on jittered sampling which we shortly introduce next. Jittered sampling in two dimension is somewhat reminiscent of the example in \cite{BS84}, where a (dependent) two-dimensional sampling scheme was constructed whose first two moments are equal to those of a Poissonian process. In our exposition we restrict ourselves to the one-dimensional case: given $N\in\N$, let $Y_1,\ldots,Y_N\sim\uniform[0,1)$, choose a random permutation $\pi:\{1,\ldots,N\}\to\{1,\ldots,N\}$ and define
\[
X_n=\frac{\pi(k)-1+Y_n}{N}
\]
for $1\leq n\leq N$. Then the (finite) sequence $X_1,\ldots,X_N$ forms the jittered sample. Note that (one-dimensional) jittered sampling is equivalent to the one-dimensional versions of two other sampling schemes, namely Latin hypercube sampling and stratified sampling, see e.g. \cite{mckay:comp_method_select_var} and \cite{quenouille:plane_sampling}.
All three sampling methods are  examples of negatively dependent sampling schemes, cf. \cite{DG21, GH21, WGH20}. \\[12pt]
As the notion of jittered sampling is restricted to finite samples, it needs to be extended to infinite sequences. We propose two different approaches to achieve this.
The first one is $M$-batch jittered sampling. We say that $\mathseq{X_{n}}{n\in\N}$ is an $M$-batch jittered sample sequence if for each $j \in \mathbb{N}_0$ the tuple (or \textit{batch}) $X^{(j)}:=\left(X_{jM+1},\ldots,X_{jM+M} \right)$ is itself a jittered sample and the random batches $X^{(1)},X^{(2)},$ $X^{(3)},\ldots$ are independent. Put differently, $(X_n)_{n \in \mathbb{N}}$ consists of successive segments of length $M$ that are independent copies of $X^{(0)}=(X_1,\ldots,X_M)$, which is itself a jittered sample.
\begin{theorem}\label{thm:batch_js_poiss_pair_corr}
    An $M$-batch jittered sampling sequence $\mathseq{X_{n}}{n\in\N}$ almost surely has PPC.
\end{theorem}
As the $M$-batch jittered sample consists of \textit{independent} copies of finite (dependent) samples, this result might have been expected by experts despite that its proof is non-trivial. Therefore, we study a second approach to extend finite jittered samples to sequences, namely sequential jittered sampling. This approach leads to a more involved dependence structure and is defined inductively. We start by creating two points using jittered sampling. Assuming that $N=2^n$ points $X_1,\ldots,X_N$ have already been constructed, we split the unit interval into $2N$ intervals 
$$\left[0,\frac{1}{2N}\right), \left[\frac{1}{2N},\frac{2}{2N}\right), \ldots, \left[\frac{2N-1}{2N},1\right).$$ 
Next, we identify those intervals void of any point so far in increasing order and call them 
$$\left[c_1,c_1+\frac{1}{2N}\right), \ldots, \left[c_N,c_N+\frac{1}{2N}\right).$$ Finally, given a random permutation $\pi$ on $\{1,\ldots,N\}$ and random points $Z_n\sim\uniform[0,1)$ we construct the new points $X_{N+1},\ldots,X_{2N}$ via
\[
    X_{N+k}=c_{\pi(k)}+\frac{Z_k}{2N}.
\]
In contrast to $M$-batch jittered samples, sequential jittered samples do not have PPC but only weak PCC.

\begin{theorem}\label{thm:seq_LHC_weak_PPC}
    Let $X=\mathseq{X_n}{n\in\N}$ be a sequential jittered sampling sequence.
    \begin{enumerate}
        \item[(i)] The sequence $X$ does not have PPC almost surely.
        \item[(ii)] In contrast, $X$ almost surely has $\alpha$-PPC for every $\alpha\in(0,1)$.
    \end{enumerate}
\end{theorem}
While it is not difficult to see that sequential jittered sampling cannot have Poissonian pair correlations by looking at the special case $s=\tfrac{1}{2}$, see Proposition~\ref{prop:seq_js_not_PPC}, we need to carefully apply point counting arguments to prove part (ii) of Theorem~\ref{thm:seq_LHC_weak_PPC}.\\

Let us conclude this introduction by briefly discussing some open questions that could be explored in future research based on the results of this paper: For example, one might ask what happens exactly if a random walk has a distribution without a Lebesgue density? Some simple examples were already discussed, but a deeper investigation of criteria for PPC would be worthwhile. For instance, it was shown in \cite[Lemma 3]{schatte:asym_distr_sum}, that the Fourier coefficients $c_r$ of a discretely distributed random variable satisfy $\sup_{r\neq 0}|c_r|=1$, meaning that we cannot use Theorem~\ref{thm:schatte_approx_unif}(2) to prove 
Theorem \ref{thm:rand_walk_ppc}.\\[12pt]
Moreover, it is open if our results extend to some of the many known generalizations of the PPC concepts which we do not consider here. For instance, a generalization of PPC which looks at $k$-tupels for $k\geq 2$ instead of pairs has been introduced in \cite{cohen:poiss_pair_higher_diff} and generalized to the weak setting in \cite{hauke:poiss_corr_higher_order}. Versions of PPC in higher dimensions have been studied e.g. in \cite{NP07,hinrichs:mult_dim_poiss-corr,bera:on_high_dim_poiss_pair_corr,Ste20}, and for the p-adic integers in \cite{Wei23}. Whether or not the ideas and techniques of the paper can be adapted to these generalizations remains open for future research.\\[12pt]
Finally, it may be possible to adapt the methods used for the proof of Theorem~\ref{thm:rand_walk_ppc} to prove similar results for more general Markov chains. This seems promising, since the long-time behavior of the sum of certain types of Markov chains on the torus has been studied in the literature, see for example \cite[Theorem 4.9]{levin:markov_chain} or \cite[Theorem 1.1]{diaconis2011geometric}.

\section{Random walk on the torus}\label{sec:rand_walk}

We will examine the behavior of a one-dimensional random walk $X=\mathseq{X_n}{n\in\N}$ on the torus in this section. The random walk starts in $x_1$ and its step size at stage $i$ is given by a random variable $Y_i$ with distribution $D$. As mentioned in the introduction, the random walk has Poissonian pair correlations, if $Y_i$ has a probability density that is in $L^p$ for some $1<p\le \infty$. Intuitively, this can be explained as follows: for $i<j$, we have $X_j-X_i=\sum_{k=i+1}^j Y_k$. If $D$ has mean $\mu$ and standard deviation $\sigma$, then $X_j-X_i$ is approximately distributed as a stretched normal distribution with parameters $(j-i)\mu$ and $\sqrt{j-i}\sigma$ according to the central limit theorem. Since we are working on the torus, i.e. calculations are made $\bmod\ 1$, this stretched normal distribution approaches the uniform distribution as illustrated in Figure \ref{fig:norm_distr_mod_1}. Moreover, it is well-known that an i.i.d. sequence of independent, uniformly distributed random variables almost surely possess Poissonian pair correlations.\\[12pt]
\begin{figure}[ht]
    \centering
    \begin{tikzpicture}[scale=2]

  \draw[->] (-2.2,0) -- (3.2,0) node[right] {$x$};
  \draw[->] (0,-0.2) -- (0,1.2);
  \def\sig{0.5}

  \draw[color=yellow,domain=-2:3,samples=200,line width=0.4mm] plot(\x,{exp(-(\x-0.5)^2/(2*\sig^2))/(sqrt(2*pi)*\sig)});
  \foreach \s in {-5,...,4}{\draw[color=red,domain=0:1] plot(\x,{exp(-(\x+\s+0.5)^2/(2*\sig^2))/(sqrt(2*pi)*\sig)});};
  \draw[color=blue,domain=0:1] plot(\x,{exp(-(\x-5.5)^2/(2*\sig^2))/(sqrt(2*pi)*\sig)+exp(-(\x-4.5)^2/(2*\sig^2))/(sqrt(2*pi)*\sig)+exp(-(\x-3.5)^2/(2*\sig^2))/(sqrt(2*pi)*\sig)+exp(-(\x-2.5)^2/(2*\sig^2))/(sqrt(2*pi)*\sig)+exp(-(\x-1.5)^2/(2*\sig^2))/(sqrt(2*pi)*\sig)+exp(-(\x-0.5)^2/(2*\sig^2))/(sqrt(2*pi)*\sig)+exp(-(\x+0.5)^2/(2*\sig^2))/(sqrt(2*pi)*\sig)+exp(-(\x+1.5)^2/(2*\sig^2))/(sqrt(2*pi)*\sig)+exp(-(\x+2.5)^2/(2*\sig^2))/(sqrt(2*pi)*\sig)+exp(-(\x+3.5)^2/(2*\sig^2))/(sqrt(2*pi)*\sig)+exp(-(\x+4.5)^2/(2*\sig^2))/(sqrt(2*pi)*\sig)+exp(-(\x+5.5)^2/(2*\sig^2))/(sqrt(2*pi)*\sig)});

  \draw[color=black] (0.05,1)--(-0.05,1) node[left] {$1$};
  \draw[color=black] (1,0.05)--(1,-0.05) node[below] {$1$};
  \draw[color=black] (0.05,0) node[below] {$0$};
\end{tikzpicture}
    \caption{The density of a normal distribution $\mathcal{N}(0.5, 0.5)$ (yellow), the multiple parts of its density mod $1$ (red) and the sum of these components, i.e. the normal distribution mod $1$ (blue).}
    \label{fig:norm_distr_mod_1}
\end{figure}
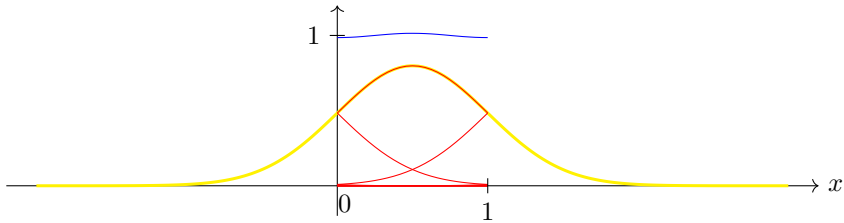

The statement that the distribution of the sum of random variables on the torus converges to the uniform distribution was first rigorously proved by Levy in \cite{levi:add_var_circ}. For our purposes we will however need a quantitative version, which depends on the Fourier coefficients of the random variables, as provided by Schatte in \cite{schatte:asym_distr_sum}.

Recall that we denote the fractional part of $x\in \R$ by $\{x\}$. 
In the following, all intervals that appear have to be understood as intervals on the torus. 
Furthermore, for a given a random variable $X_i$ with a Lebesgue density, we will write $\delta_{X_i}$ for its density.

\begin{theorem}{\cite[Lemma 2, Theorem 3]{schatte:asym_distr_sum}}\label{thm:schatte_approx_unif}
    Let $Y$ be a random variable on $[0,1)$, and let $Y_1,Y_2, \ldots$ be independent copies of $Y$. Let $\mathseq{c_r}{r\in\Z}$ be the Fourier coefficients of $Y$ defined by 
    \[
      c_r = \int_0^1 e^{2\pi irx} \mathrm{d}F(x),
    \]
    where $F(\cdot)$ is the cumulative distribution function of $Y$. Then the following statements hold:
    \begin{enumerate}
        \item[(i)] If $Y$ has an integrable Lebesgue density, then $$\sup_{r\neq 0}\vert c_r\vert<1.$$
        \item[(ii)] If $\sup_{r\neq 0}\vert c_r\vert<1$, then there exist $C\geq 0$ and $\omega<1$ such that for all $n\in \N$ the cumulative distribution function $G_n(x)$ of $\{Y_1+\ldots+Y_n\}$ satisfies
\[
    \sup_{0\leq x<1}\vert G_n(x)-x\vert\leq C\omega^n.
\]    
    \end{enumerate}
\end{theorem}

\iftrue
Theorem~\ref{thm:schatte_approx_unif} gives a quantitative bound on how much the sum of the $Y_i$ may deviate from a uniformly distributed random variable. Note that it cannot be applied directly to the pair correlation statistic $R(s,N)$ because the deviation is comparably large for small $n$. Although it will nevertheless be the basis for the proof of Theorem~\ref{thm:rand_walk_ppc}, some extra work is therefore required. We start our preparations with the next lemma.

\begin{lemma} \label{lem:pre_randomsum}
Let $Y$ be a random variable on $[0,1)$ with probability density $\delta_Y\in L^p([0,1))$ for some $p\in [1, \infty]$, 
and let $Y_1,Y_2, \ldots$ be independent copies of $Y$. Then we have for every measurable set $A\subseteq [0,1)$ that
\begin{equation}\label{eq:basic_prob_est_pre_1}
\Prob{\{Y_1+ \cdots + Y_n \}\in A } \le \|\delta_Y\|_p \cdot \lambda(A)^{1-1/p}
\end{equation}
Now assume $p <\infty$. We denote the probability density of the sum on the torus $\{Y_1+ \cdots + Y_n\}$ by $\delta_n$. Then the family of functions  $(|\delta_n|^p)_{n\in\N}$ is uniformly integrable. For $\varepsilon \in (0,1]$ let $A(\varepsilon)$ be a  measurable set with Lebesgue measure 
$\lambda(A(\varepsilon)) = \varepsilon$. Then we obtain, in particular,
\begin{equation}\label{eq:basic_prob_est_pre_2}
 \Prob{\{Y_1+ \cdots + Y_n \}\in A(\varepsilon) } =  o( \varepsilon^{1-1/p}).
\end{equation}
In the last estimate the implicit constant in the small-$o$-notation does not depend on $n$, $\varepsilon$ or the specific
choice of $A(\varepsilon)$.
\end{lemma}
\begin{proof}
For every  $n\in\N$ the probability density $\delta_n$ of $\{Y_1 + Y_2 + \cdots + Y_n \}$
is simply the $n$-fold convolution $\delta_Y \ast \cdots \ast \delta_Y$ (taken over the torus with respect to addition modulo $1$) of $\delta_Y$, that is 
$\delta_n = \delta_Y \ast \delta_{n-1}$. 
Using H\"older's inequality, we obtain for any measurable set $A$ that
\begin{equation}\label{inequality_Holder}
\begin{split}
\Prob{\{Y_1+ \cdots + Y_n \}\in A } &= \int_{0}^{1}  \delta_n(x)\ind_A(x) \, {\rm d}x = \int_{0}^{1}  (\delta_n(x) \ind_A(x))\ind_A(x)\, {\rm d}x\\
&\leq \| \delta_n \ind_A \|_p \left( \int_{0}^{1} \ind_A(x) \,{\rm d}x \right)^{1-1/p}= \| \delta_n \ind_A \|_p \lambda(A)^{1-1/p}.\\
\end{split}
\end{equation}
We observe next that
\begin{equation}\label{12_strich}
\begin{split}
\|\delta_n \ind_A \|_{p} 
&= \left\| \int_{0}^{1}  \left( \delta_Y(\cdot - y) \ind_A \right) \delta_{n-1}(y) \, {\rm d}y \right\|_p 
\le  \int_{0}^{1} \left\| \delta_Y(\cdot - y) \ind_A \right\|_p |\delta_{n-1}(y)| \, {\rm d}y  \\
&= \int_{0}^{1} \left\| \delta_Y \ind_{A-y}\right\|_p |\delta_{n-1}(y)| \, {\rm d}y,
\end{split}
\end{equation}
where $A-y$ denotes the set $\mathset{a-y}{a\in A}$. 
Moreover, we get for all $y\in [0,1)$
\begin{equation}\label{trivial_bound_A}
\left\| \delta_Y \ind_{A-y} \right\|_p   \le \left\| \delta_Y \right\|_p.
\end{equation}
Since $\delta_{n-1}$ is a probability measure, the estimates \eqref{inequality_Holder}, \eqref{12_strich}, and \eqref{trivial_bound_A}
yield inequality \eqref{eq:basic_prob_est_pre_1}.\\
Now let $p<\infty$. Then the function $\delta^p_Y$ is Lebesgue integrable, hence the trivial family just consisting
of $\delta^p_Y$ is uniformly integrable, implying that for any $\varepsilon >0$ there exists a $\sigma = \sigma(\varepsilon)$
such that for all measurable sets $B$ with $\lambda(B) \le \sigma$ we have $\int_B |\delta_Y(x)|^p \,{\rm d}x < \varepsilon$,
see, e.g., \cite[Thm.~6.18 \& 6.24]{KlenkePT}. For $A=A(\varepsilon)$ this yields  $\int_{A(\varepsilon)} |\delta_Y(x)|^p \,{\rm d}x = o(1)$
as $\varepsilon$ tends to $0$.
Hence we get 
\begin{equation}\label{uniformly_integrable_A-y}
\left\| \delta_Y 1_{A-y} \right\|_p^p
= \int_{A-y} |\delta_Y(x)|^p \, {\rm d}x = o(1),
\end{equation}
uniformly in $y\in [0,1)$, since $\lambda(A-y) = \lambda(A) = \varepsilon$.
In particular, estimates \eqref{12_strich}
and  \eqref{uniformly_integrable_A-y} show that the uniform integrability of $\delta_Y^p$ implies the uniform integrability
of the whole family $\left( \delta_n^p \right)_{n\in\N}$.
Moreover, since $\delta_{n-1}$ is a probability measure, the estimates  \eqref{inequality_Holder}, \eqref{12_strich}, and \eqref{uniformly_integrable_A-y}
yield estimate \eqref{eq:basic_prob_est_pre_2}.
\end{proof}

\begin{corollary} \label{lem:randomsum}
Let $Y$ be a random variable on $[0,1)$ with probability density $\delta_Y \in L^p([0,1))$ for some $1\le p \le \infty$, and let $Y_1,Y_2, \ldots$ be independent copies of $Y$. Let $s_0>0$ be arbitrary and $s \in [0,s_0)$. Then there exists an $\omega \in (0,1)$ such that for every $n, N\in\N$ with $s/N \in (0,1/2)$ and every $x\in [0,1)$, 
\begin{equation}\label{eq:basic_prob_est}
\Prob{\{Y_1+ \cdots + Y_n \}\in \left[\ x - \frac{s}{N} , x + \frac{s}{N}\right]} = \frac{2s}{N} + \mathcal{O} \left( \omega^{n} \right).
\end{equation}
Furthermore, we have
\begin{equation}\label{eq:basic_prob_est_2}
 \Prob{\{Y_1+ \cdots + Y_n \}\in \left[\ x - \frac{s}{N} , x + \frac{s}{N}\right]} = \mathcal{O} \left( \frac{1}{N^{1-1/p}}\right).
\end{equation}
If $p=1$ we still have
\begin{equation}\label{eq:basic_prob_est_3}
 \Prob{\{Y_1+ \cdots + Y_n \}\in \left[\ x - \frac{s}{N} , x + \frac{s}{N}\right]} = o \left( 1 \right).
\end{equation}
In all three estimates the constants in the $\mathcal{O}$- and $o$-notation do not depend on $n$, $N$, $x$, or $s$.
\end{corollary}
To avoid too cumbersome notation, let us from now on use the short hands 
\begin{equation} \label{eq:abbreviations}
\begin{split}
P_{Z_1}(x) &:=  \Prob{\{Z_1\}\in \left[ - \frac{s}{N}-x, \frac{s}{N}-x \right]}, \\
P_{Z_1,Z_2}(x) &:=  \Prob{\{Z_2\}\in \left[ - \frac{s}{N}-x, \frac{s}{N}-x\right] \wedge \{Z_1\}\in \left[- \frac{s}{N}-x, \frac{s}{N}-x \right]} \\
\end{split}
\end{equation}
for arbitrary random variables $Z_1, Z_2$, and let us suppress the argument $x\in [0,1)$ in case that $x=0$.
\begin{proof} 
Let $G_n(x)$ denote the cumulative distribution function  of $\{Y_1+\ldots+Y_n\}$. In the case where
$0 \le \{x - \frac{s}{N}\} < \{x + \frac{s}{N}\} <1$, we actually have $0 \le x - \frac{s}{N} < x + \frac{s}{N} <1$, because $s/N<1/2$,
and therefore obtain from Theorem~\ref{thm:schatte_approx_unif}
\begin{equation*}
\begin{split}
&P_{Y_1+ \cdots + Y_n}(x) = G_n\left( x + \frac{s}{N} \right) - G_n\left( x - \frac{s}{N} \right)\\
&= x + \frac{s}{N} + \left( G_n\left( x + \frac{s}{N} \right) - \left(x + \frac{s}{N}\right) \right) 
- \left( x - \frac{s}{N} + \left( G_n\left( x - \frac{s}{N} \right) - \left(x - \frac{s}{N}\right) \right) \right)\\
&\in\left[\frac{2s}{N} -2C \omega^{n},\frac{2s}{N} +2C \omega^{n}\right]
\end{split}
\end{equation*}
for some $\omega \in (0,1)$ and some $C>0$, both independent of $n$, $N$, $x$, and $s$.
The second case $0 \le \{x + \frac{s}{N}\} < \{x - \frac{s}{N}\} <1$ can be handled similarly. 

Formula \eqref{eq:basic_prob_est_2} follows immediately from estimate \eqref{eq:basic_prob_est_pre_1}, and
formula \eqref{eq:basic_prob_est_3} is a consequence of formula \eqref{eq:basic_prob_est_pre_2}.
\end{proof}

Note that Corollary~\ref{lem:randomsum} can (and will) also be applied in the case $n=N$.

\begin{proof}[Proof of Theorem \ref{thm:rand_walk_ppc}] 
One of the final steps in the proof will be to apply \\Chebyscheff's~inequality. This is a typical trick to prove the PPC property for random sequences and has been applied several times in the literature, see e.g. \cite[Theorem 1]{rudnick:pair_corr_polynom} and \cite[Theorem 1]{hinrichs:mult_dim_poiss-corr}. Therefore, we will calculate the expected value and the variance of 
    \begin{equation} \label{eq:RsN}
        R(s,N)=\frac{1}{N}\sum_{1\leq i\neq j\leq N}\ind_{\left\{\norm{X_i-X_j}\leq \frac{s}{N}\right\}}
    \end{equation}
for $s > 0$ first in order to show that $\lim_{N\to\infty}R(s,N)=2s$ holds almost surely.\\[12pt]
Now we consider a monotone increasing function $\nu:\N\to \N_0$ with $\lim_{N\to \infty} \nu(N) = \infty$. In the case where $p>1$, we define $\nu$ by $\nu(N) = \lceil -\ln( N )/\ln(\omega) \rceil$ with $\omega < 1$ as in Corollary~\ref{lem:randomsum}.
In the case where $p=1$, we choose $\nu$ in such a way, that the product of $\nu(N)$ and the left hand side of \eqref{eq:basic_prob_est_3} goes to zero if $N$ tends to infinity, i.e., (with some abuse of notation) $\lim_{N\to \infty} \left( \nu(N)\cdot o_N(1) \right) = 0$, where the subscript in $o_N(1)$ indicates that the implicit constant is independent of $N$.
Furthermore, for $N\in \N$, $i\leq N$ we define
\[
\mu_{i,N}:=\min\{ i+\nu(N),N\}.
\]
Note that 
\[
\Exp{R(s,N)} = S_{1,N}+ S_{2,N}, 
\]
where 
\begin{equation*}
\begin{split}
S_{1,N} &= \frac{2}{N} \sum_{i=1}^N \sum_{j=i+1}^{\mu_{i,N}} \Prob{\norm{X_j-X_i}\leq\frac{s}{N}},\\
S_{2,N} &= \frac{2}{N} \sum_{i=1}^N \sum_{j= \mu_{i,N} + 1}^N \Prob{\norm{X_j-X_i}\leq\frac{s}{N}}.
\end{split}
\end{equation*}
In the next step we establish $\lim_{N\to\infty} S_{2,N} = 2s$ and afterwards $\lim_{N\to\infty} S_{1,N} = 0$. With
\begin{equation}\label{geom_sum_trick}
\begin{split}
    \frac{2}{N}\sum_{i=1}^N\sum_{j=\mu_{i,N}+1}^N\omega^{j-i}&\leq \frac{2}{N}\sum_{i=1}^N\omega^{\nu(N) +1}\sum_{j=0}^{N-\mu_{i,N}-1}\omega^{j}\\
    &\leq \frac{2}{N}N\omega^{\nu(N)+1}\frac{1}{1-\omega}=\frac{2\omega^{\nu(N)+1}}{1-\omega}
\end{split}
\end{equation}
and \eqref{eq:basic_prob_est} we get
\begin{equation*}
\begin{split}
    S_{2,N} &= \frac{2}{N}\sum_{i=1}^N\sum_{j=\mu_{i,N}+1}^{N}\Prob{\{Y_{i+1} + \cdots + Y_j\} \in \left[ -\frac{s}{N}, \frac{s}{N} \right] }\\
    &=\frac{2}{N}\sum_{i=1}^N\sum_{j=\mu_{i,N}+1}^{N}\frac{2s}{N}+\mathcal{O}_N\left(\omega^{\nu(N)}\right).
\end{split}
\end{equation*}
Since $\omega^{\nu(N)} \to 0$ for $N\to \infty$, we obtain 
\begin{equation*}
    \lim_{N\to\infty} S_{2,N} =\lim_{N\to\infty}\frac{2}{N}\sum_{i=1}^N\sum_{j=\mu_{i,N}+1}^{N}\frac{2s}{N}=2s.
\end{equation*}
We now establish $\lim_{N\to\infty} S_{1,N} = 0$. In the case where $p>1$, inequality~\eqref{eq:basic_prob_est_2} implies
\begin{align*}
    S_{1,N} &= \frac{2}{N}\sum_{i=1}^N\sum_{j=i+1}^{\mu_{i,N}}\Prob{\norm{X_j-X_i}\leq\frac{s}{N}}\\
    &\leq\frac{2}{N}N\nu(N) \cdot \mathcal{O}_N \left( \frac{1}{N^{1-1/p}} \right)\\
    &=\mathcal{O}_N\left(\frac{\ln(N)}{N^{1-1/p}}\right).
\end{align*}
In the case where $p=1$, estimate \eqref{eq:basic_prob_est_3} implies
\begin{align*}
    S_{1,N} &\leq\frac{2}{N}N\nu(N) \cdot o_N \left( 1 \right) = 2 \nu(N) \cdot o_N \left( 1 \right).
\end{align*}
Hence our case specific choices of $\nu$ always yield $\lim_{N\to\infty} S_{1,N} = 0$.

Combining the results above we get for all $1\le p \le \infty$ the desired identity
\[
\lim_{N\to \infty} \Exp{R(s,N)} = \lim_{N\to\infty} S_{1,N} + \lim_{N\to\infty} S_{2,N} = 2s.
\]
Now we confine ourselves to the case $1<p\le \infty$ and study the variance 
\[
\var(R(s,N))= \Exp{R(s,N)^2}-\left(\mathbb{E}\left( R(s,N) \right) \right)^2.
\]
By using the representation \eqref{eq:RsN} of $R(s,N)$ and the definitions of $P_{Z_1}$ and $P_{Z_1,Z_2}$ for arbitrary random variables $Z_1$ and $Z_2$ in \eqref{eq:abbreviations}, we obtain
\begin{align}\label{eq:quad_sum_rand_walk}
\var(R(s,N)) =\frac{4}{N^2}\sum_{i=1}^N\sum_{j=i+1}^N\sum_{l=1}^N\sum_{k=l+1}^N
    \left( P_{X_j-X_i,X_k - X_l} - P_{X_j-X_i}P_{X_k - X_l} \right).
\end{align}
We will now take a look at three different regimes of the relation between the indices $i,j,k,l$ and show that the  corresponding (sub-)sums  
in \eqref{eq:quad_sum_rand_walk} are of order $\mathcal{O}_N(\ln(N)/N^{1-1/p})$. Since $i<j$ and $l<k$ and since formula \eqref{eq:quad_sum_rand_walk} 
is symmetric with respect to the pairs $(i,j)$ and $(l,k)$,
we only need to consider three cases, namely  
\begin{enumerate}
\item $i< j \le l < k$, 
\item $i \le l < j \le k$,
\item $i < l < k < j$.
\end{enumerate}

\underline{Case 1}: $i< j\leq l <k$. Now the stochastic independence of the random variables $Y_1, Y_2, \ldots$ implies that 
$X_j-X_i = Y_{i+1} + \cdots + Y_j$ and $X_k-X_l = Y_{l+1} + \cdots + Y_k$ are independent, and therefore it follows 
that 
\begin{equation*}
P_{X_j-X_i,X_k - X_l} =  P_{X_j-X_i}P_{X_k - X_l},
\end{equation*}
yielding that the corresponding subsum in \eqref{eq:quad_sum_rand_walk} is even zero.

\underline{Case 2}: $i \le l < j \le k $. We consider the subsum 
\begin{equation}\label{subsum_case_2}
     \frac{4}{N^2}\underset{i \le l < j \le k}{\sum_{i=1}^N\sum_{j=i+1}^N\sum_{l=1}^N\sum_{k= l+1}^N} \left( P_{X_j-X_i,X_k - X_l} - P_{X_j-X_i}P_{X_k - X_l} \right)
\end{equation}
of \eqref{eq:quad_sum_rand_walk}.
Note that in this case, due to the order of the indices, $X_l-X_i = Y_{i+1} + \cdots + Y_l$, $X_j-X_l = Y_{l+1} + \cdots + Y_j$, and $X_k-X_j = Y_{j+1} + \cdots + Y_k$ are independent random variables. We use $X_j-X_i = (X_j-X_l) + (X_l-X_i)$ and $X_k-X_l = (X_k-X_j) + (X_j-X_l)$ to obtain 
\begin{align}
\label{eq:int_density_prob}
\begin{split}
    P_{X_j-X_i, X_k-X_l} &= P_{X_l - X_i, X_k - X_j}(X_j - X_l)\\
    &=\int_{-\frac{1}{2}}^{\frac{1}{2}} P_{X_l - X_i, X_k - X_j}(x) \,\delta_{X_j-X_l}(x) \diff x\\
    &=\int_{-\frac{1}{2}}^{\frac{1}{2}} P_{X_l - X_i}(x) P_{ X_k - X_j}(x)\, \delta_{X_j-X_l}(x) \diff x.
\end{split}
\end{align}
Note that because of \eqref{eq:basic_prob_est_2} and the fact 
that $\delta_{X_j-X_l}$ is a probability density, we may always resort to the estimates
\begin{equation}\label{resort_estimates}
  P_{X_j-X_i}P_{X_k-X_l} = \mathcal{O}_N \left( \frac{1}{N^{2-2/p}} \right) 
  \hspace{3ex}\text{and}\hspace{3ex}
  P_{X_j-X_i, X_k-X_l} = \mathcal{O}_N \left( \frac{1}{N^{2-2/p}} \right).
\end{equation}
Consequently, any subsum of \eqref{subsum_case_2} 
that contains $\mathcal{O}_N\left( \left( \nu(N)\right)^2 N^2\right)$ terms
is already of order $\mathcal{O}_N\left(  \left( \frac{\nu(N)}{N^{1-1/p}}\right)^2 \right)$. 
Let us first consider the subsum 
\begin{equation*}
    S_a:= \frac{4}{N^2}\sum_{i=1}^N\sum_{l=\mu_{i,N}+1}^N\sum_{j=l+1}^N\sum_{k=\mu_{j,N}+1}^N
    \left( P_{X_j-X_i,X_k - X_l} - P_{X_j-X_i}P_{X_k - X_l} \right).
\end{equation*}
of \eqref{subsum_case_2}.
Notice that in this case, due to the order of the indices, we have
\[
\frac{\omega^{j-i}}{N}  +  \frac{\omega^{k-l}}{N} +  \omega^{j-i + k -l}
= \mathcal{O}_N \left( \frac{\omega^{l-i}}{N} \right) + \mathcal{O}_N \left( \frac{\omega^{k-j}}{N} \right) + \mathcal{O}_N \left( \omega^{l-i +k-j } \right).
\]
Hence, by applying \eqref{eq:basic_prob_est} to $P_{X_j - X_i}$, $P_{X_k - X_l}$, and to the integral representation of $P_{X_j - X_i, X_k - X_l}$ in
\eqref{eq:int_density_prob}, we obtain
\begin{equation*}
  S_a = \frac{4}{N^2}\sum_{i=1}^N\sum_{l=\mu_{i,N}+1}^N\sum_{j=l+1}^N\sum_{k=\mu_{j,N}+1}^N
    \left( \mathcal{O}_N\left( \frac{\omega^{l-i}}{N} \right) + \mathcal{O}_N\left( \frac{\omega^{k-j}}{N} \right) + \mathcal{O}_N\left( \omega^{l-i + k-j} \right) \right).
\end{equation*}
By additionally using the estimate \eqref{geom_sum_trick} three times (for different pairs of indices), we then get due to the definition of $\nu(N)$
\begin{equation*}
\begin{split}    
    S_a &= \frac{4}{N^2}\sum_{i=1}^N\sum_{l=\mu_{i,N}+1}^N\sum_{j=l+1}^N
    \left( \mathcal{O}_N \left( \omega^{l-i}\right) + \mathcal{O}_N \left( \frac{\omega^{\nu(N)}}{N} \right) + \mathcal{O}_N \left( \omega^{l-i + \nu(N)} \right) \right)\\
    &= \frac{4}{N^2}\sum_{i=1}^N\sum_{l=\mu_{i,N}+1}^N
    \left( \mathcal{O}_N \left( N \omega^{l-i}\right) + \mathcal{O}_N \left( \omega^{ \nu(N)} \right) \right)\\
    &= \frac{4}{N^2}\sum_{i=1}^N \mathcal{O}_N \left( N \omega^{ \nu(N)} \right)\\
    &= \mathcal{O}_N \left(\omega^{ \nu(N)} \right)\\
    &= \mathcal{O}_N \left(\frac{1}{N} \right).
\end{split}    
\end{equation*}
Next we consider the subsum
\begin{equation*}
    S_b:= \frac{4}{N^2}\sum_{i=1}^N\sum_{l=\mu_{i,N}+1}^N\sum_{j=l+1}^N\sum_{k=j}^{\mu_{j,N}}
    \left( P_{X_j-X_i,X_k - X_l} - P_{X_j-X_i}P_{X_k - X_l} \right)
\end{equation*}
of \eqref{subsum_case_2}. Here we employ the estimates 
\begin{equation*}
P_{X_j-X_i} = \frac{2s}{N} + \mathcal{O}_N (\omega^{j-i})
\hspace{3ex}\text{and}\hspace{3ex}
P_{X_k-X_l} = \mathcal{O}_N \left( \frac{1}{N^{1-1/p}} \right)
\end{equation*}
to estimate the products $P_{X_j-X_i}P_{X_k - X_l}$ and the estimates
\begin{equation*}
P_{X_l-X_i} = \frac{2s}{N} + \mathcal{O}_N (\omega^{l-i})
\hspace{3ex}\text{and}\hspace{3ex}
P_{X_k-X_j} = \mathcal{O}_N \left( \frac{1}{N^{1-1/p}} \right)
\end{equation*}
in combination with \eqref{eq:int_density_prob} to estimate the probabilities 
$P_{X_j-X_i, X_k-X_l}$. Since $j-i > l-i$, this implies 
\begin{equation*}
\begin{split}
    S_b&= \frac{4}{N^2}\sum_{i=1}^N\sum_{l=\mu_{i,N}+1}^N\sum_{j=l+1}^N\sum_{k=j}^{\mu_{j,N}}
    \left( \mathcal{O}_N\left( \frac{1}{N^{2-1/p}} \right) + \mathcal{O}_N \left( \frac{\omega^{l-i}}{N^{1-1/p}} \right) \right)\\
    &= \mathcal{O}_N \left( \frac{\nu(N)}{ N^{1-1/p}} \right) + \frac{4}{N^2}\sum_{i=1}^N\sum_{l=\mu_{i,N}+1}^N 
    \mathcal{O}_N \left( \nu(N) N^{1/p} \omega^{l-i} \right)\\
    &= \mathcal{O}_N \left( \frac{\nu(N)}{ N^{1-1/p}} \right) + \mathcal{O}_N \left( \frac{\nu(N)}{ N^{1-1/p}} \omega^{\nu(N)}\right)\\
    &= \mathcal{O}_N \left( \frac{\nu(N)}{ N^{1-1/p}} \right).
\end{split}    
\end{equation*}
Similarly, we can show that also the second subsum 
\begin{equation*}
    S_c:= \frac{4}{N^2}\sum_{i=1}^N\sum_{l=i}^{\mu_{i,N}}\sum_{j=l+1}^N\sum_{k= \mu_{j,N}+1}^N
    \left( P_{X_j-X_i,X_k - X_l} - P_{X_j-X_i}P_{X_k - X_l} \right).
\end{equation*}
with $\mathcal{O}\left(\nu(N)N^3\right)$ summands satisfies $S_c = \mathcal{O}_N \left( \frac{\nu(N)}{ N^{1-1/p}} \right)$;
this time we rely on the estimates
\begin{equation*}
P_{X_j-X_i} = \mathcal{O}_N \left( \frac{1}{N^{1-1/p}} \right)
\hspace{3ex}\text{and}\hspace{3ex}
P_{X_k-X_l} = \frac{2s}{N} + \mathcal{O}_N (\omega^{k-l})
\end{equation*}
to estimate the products $P_{X_j-X_i}P_{X_k - X_l}$ and the estimates
\begin{equation*}
P_{X_l-X_i} =  \mathcal{O}_N \left( \frac{1}{N^{1-1/p}} \right)
\hspace{3ex}\text{and}\hspace{3ex}
P_{X_k-X_j} = \frac{2s}{N} + \mathcal{O}_N (\omega^{k-j})
\end{equation*}
in combination with \eqref{eq:int_density_prob} to estimate the probabilities 
$P_{X_j-X_i, X_k-X_l}$.

\underline{Case 3}: $i < l <  k < j $.
Put $Z:= X_j-X_k + X_l - X_i$. Then $X_j-X_i= Z + (X_k - X_l)$ holds, and furthermore $Z = Y_{i+1} + \cdots + Y_l + Y_{k+1} + \cdots + Y_j$ 
and $X_k-X_l = Y_{l+1} + \cdots +Y_k$ are stochastically independent. We have
\begin{equation}\label{eq:int_density_prob_2}
\begin{split}
    &P_{X_j-X_i, X_k-X_l}\\
    = &\Prob{\{Z\} \in\left[-\frac{s}{N} - (X_k-X_l), \frac{s}{N} - (X_k-X_l) \right]\wedge X_k-X_l\in\left[-\frac{s}{N},\frac{s}{N} \right]}\\
    = &\int_{-\frac{s}{N}}^{\frac{s}{N}} P_{Z}(x) \,\delta_{X_k-X_l}(x) \diff x.
\end{split}
\end{equation}
Again, due to \eqref{eq:basic_prob_est_2} and the fact that $\delta_{X_j-X_l}$ is a probability density, we can always resort to the estimates \eqref{resort_estimates}.
Consequently, any subsum of \eqref{subsum_case_2} 
that contains $\mathcal{O}_N\left( \left( \nu(N)\right)^2 N^2\right)$ terms
is of order $\mathcal{O}_N\left( \left( \frac{\nu(N)}{N^{1-1/p}}\right)^2  \right)$, exactly as in Case 2.
So let us consider the subsum
\begin{equation*}
    T_a := \frac{4}{N^2}\sum_{i=1}^N\sum_{l=\mu_{i,N}+1}^N\sum_{k=\mu_{l,N}+1}^N\sum_{j=k+1}^N
    \left( P_{X_j-X_i,X_k - X_l} - P_{X_j-X_i}P_{X_k - X_l} \right).
\end{equation*}
Note that the order of the indices implies
\[
\frac{\omega^{j-i}}{N} + \omega^{j-i+k-l} = \mathcal{O}_N \left( \omega^{j-i} \right).
\]
Hence, by applying \eqref{eq:basic_prob_est} to the integral in \eqref{eq:int_density_prob_2} as well as to $P_{X_j-X_i}$ and $P_{X_k - X_l}$, we get\begin{equation*}
\begin{split}
   T_a = &\frac{4}{N^2}\sum_{i=1}^N\sum_{l=\mu_{i,N}+1}^N\sum_{k=\mu_{l,N}+1}^N\sum_{j=k+1}^N
    \left( \mathcal{O}_N\left( \frac{\omega^{l-i+ j-k}}{N} \right) + \mathcal{O}_N\left( \frac{\omega^{k-l}}{N} \right) + \mathcal{O}_N\left( \omega^{j-i} \right)
    \right)\\
  = &\frac{4}{N^2}\sum_{i=1}^N\sum_{l=\mu_{i,N}+1}^N\sum_{k=\mu_{l,N}+1}^N
    \left( \mathcal{O}_N\left( \frac{\omega^{l-i}}{N} \right) + \mathcal{O}_N\left( \omega^{k-l} \right) + \mathcal{O}_N\left( \omega^{k-i} \right)
    \right)\\  
    = &\frac{4}{N^2}\sum_{i=1}^N\sum_{l=\mu_{i,N}+1}^N
    \left( \mathcal{O}_N\left( \omega^{l-i} \right) + \mathcal{O}_N\left( \omega^{\nu(N)} \right)  \right)\\  
    = &\frac{4}{N^2}\sum_{i=1}^N
    \left( \mathcal{O}_N\left( \omega^{\nu(N)} \right) + \mathcal{O}_N\left( N \omega^{\nu(N)} \right)  \right)\\ 
    = &  \mathcal{O}_N\left( \omega^{\nu(N)} \right)\\
    = &\mathcal{O}_N \left(\frac{1}{N} \right).
\end{split}
\end{equation*}
Similarly as in Case~2, we can show that the two remaining sums with $\mathcal{O}_N\left(\nu(N)N^3\right)$ summands,
\begin{equation*}
    T_b:= \frac{4}{N^2}\sum_{i=1}^N\sum_{l=\mu_{i,N}+1}^N\sum_{k=l+1}^{\mu_{l,N}} \sum_{j=k+1}^N
    \left( P_{X_j-X_i,X_k - X_l} - P_{X_j-X_i}P_{X_k - X_l} \right).
\end{equation*}
and
\begin{equation*}
    T_c:= \frac{4}{N^2}\sum_{i=1}^N\sum_{l=i+1}^{\mu_{i,N}}\sum_{k= \mu_{l,N}+1}^N \sum_{j=k+1}^N
    \left( P_{X_j-X_i,X_k - X_l} - P_{X_j-X_i}P_{X_k - X_l} \right),
\end{equation*}
are both of order $\mathcal{O}_N \left( \frac{\nu(N)}{ N^{1-1/p}} \right)$.

By considering all three cases above, we have altogether shown that
$$
\var(R(s,n)) = 
\mathcal{O}_N\left(\frac{1}{N}\right)
+\mathcal{O}_N\left(\frac{\ln(N)}{N^{1-1/p}}\right) = \mathcal{O}_N\left(\frac{\ln(N)}{N^{1-1/p}}\right).
$$
 Using Chebyscheff's inequality results in
\[
\Prob{\vert R(s,N)-\Exp{R(s,N)}\vert\geq \frac{\sqrt{\ln(N)}}{N^{(1-1/p)/4}}}  = \mathcal{O}_N \left( N^{-(1-1/p)/2}\right).
\]
By the first Borel-Cantelli Lemma it follows for the subsequence 
$$\seq{X_{N^{\gamma}}}{N\in\N}$$ 
with $\gamma:=\left\lceil\frac{4}{1-1/p}\right\rceil$ that $\lim_{N\to\infty}R(s,N^\gamma)=2s$ almost surely. In order to show that this is also true for the entire sequence, we use the mentioned standard trick, which can for example be found in the proof of \cite[Theorem 1]{rudnick:pair_corr_polynom} or \cite[Theorem 1]{hinrichs:mult_dim_poiss-corr}. For every $k\in\N$ and every $N_0 \in \mathbb{N}$ with 
$$N^\gamma \le N_0<(N+1)^\gamma $$ 
we have
    \begin{align*}
        N^\gamma R\left(\frac{N^\gamma}{(N+1)^\gamma}s,N^\gamma\right) \leq N_0 R(s,N_0)\leq (N+1)^\gamma R\left(\frac{(N+1)^\gamma}{N^\gamma}s,(N+1)^\gamma\right).
    \end{align*}
As $\lim_{k\to\infty}\frac{N^\gamma}{(N+1)^\gamma}=\lim_{k\to\infty}\frac{(N+1)^\gamma}{N^\gamma}=1$, this implies the almost sure convergence of the entire sequence, which completes the proof.
\end{proof}
\fi

\section{Jittered sampling}\label{sec:JS}
In this section, we treat the sequences derived from jittered sampling. As we have already described in the introduction, we consider two different types of such sequences, namely $M$-batch jittered sampling and sequential sampling. Our first aim is the proof of Theorem~\ref{thm:batch_js_poiss_pair_corr}. The approximation techniques which we will use are inspired by the proof of \cite[Theorem 1.2]{lachman:additive_energy_discrepancy_poissonian}. However, we stress here that \cite[Theorem 1.2]{lachman:additive_energy_discrepancy_poissonian} is not directly applicable to $M$-batch jittered sample.
\begin{proof}[Proof of Theorem \ref{thm:batch_js_poiss_pair_corr}] 
We proceed similarly as in the proof of Theorem \ref{thm:rand_walk_ppc}. Note that if $M=1$, the sequence is i.i.d. and therefore has PPC, we will therefore from now on assume that $M\geq 2$.
    First, we will look at the subsequence $\mathseq{X_{kM}}{k\in\N}$. 
    Let $N=kM$ for some $k\in\N$ and assume that $k$ is large enough that $\frac{s}{N}<\frac{1}{M}$. Let $1\leq i,j\leq N$ and $i\neq j$. We can write $i=k_iM+l_i$ and $j=k_jM+l_j$ with $0\leq l_1,l_2<M$ and $k_{1},k_2 \in \mathbb{N}$. Let us set
    \[
        \mathcal{M}(i):=\{k_iM+1,k_iM+2,\ldots,(k_i+1)M\}.
    \]
    If $j\notin\mathcal{M}(i)$, then $X_i$ and $X_j$ are independent and we have \begin{equation}\label{eq:eq_Exp_Xi_Xj_sN}
        \Exp{\mathds{1}_{\norm{X_i-X_j}\leq \frac{s}{N}}}= \frac{2s}{N}.
    \end{equation}
    If $j\in\mathcal{M}(i)$, we will now show
    \begin{equation}\label{eq:ineq_Exp_Xi_Xj_sN}
        \Exp{\mathds{1}_{\norm{X_i-X_j}\leq \frac{s}{N}}}\leq \frac{2s}{N}.
    \end{equation}
    If $\norm{X_i-X_j} \leq \frac{s}{N}$, then the two points must be in two adjacent subintervals. In this case we may without loss of generality assume that $X_i\in I:=\left[\frac{l}{M},\frac{l+1}{M}\right]$ holds for some $0 \leq l \leq M$. Then $X_j$ is uniformly distributed on $[0,1)\setminus I$, meaning it has density $\frac{M}{M-1}\ind_{[0,1)\setminus I}$. Since $\frac{s}{N}<\frac{1}{M}$, we have
    \begin{align*}
        \Exp{\ind_{\norm{X_i-X_j}\leq\frac{s}{N}}}&=\frac{M}{M-1} \mathbb{E} \left( \leb\left(\mathset{x\in[0,1)}{\norm{x-X_i}\leq\frac{s}{N}}\setminus I\right) \right)\\
        &\leq 2\mathbb{E} \left(\leb\left(\mathset{x\in[0,1)}{\norm{x-X_i}\leq\frac{s}{N}}\setminus I\right)\right)\\
        &\leq \mathbb{E} \left(\leb\left(\mathset{x\in[0,1)}{\norm{x-X_i}\leq\frac{s}{N}}\right)\right)=\frac{2s}{N}.
    \end{align*}
    This proves \eqref{eq:ineq_Exp_Xi_Xj_sN}.\\
    Hence,
    \begin{align*}
        \Exp{R(s,N)}&=\frac{1}{N}\sum_{i=1}^N\sum_{\substack{j=1\\j\neq i}}^N\Exp{\mathds{1}_{\norm{X_i-X_j}\leq \frac{2s}{N}}}\\
        &=\frac{1}{N}\sum_{i=1}^N\underbrace{\sum_{\substack{j=1\\j\notin \mathcal{M}(i)}}^N\Exp{\mathds{1}_{\norm{X_i-X_j}\leq \frac{s}{N}}}}_{=(k-1)M\frac{2s}{N}}+\frac{1}{N}\sum_{i=1}^N\underbrace{\sum_{\substack{j=1\\j\in \mathcal{M}(i)\setminus \{i\}}}^N\Exp{\mathds{1}_{\norm{X_i-X_j}\leq \frac{s}{N}}}}_{\leq M\frac{2s}{N}}.
    \end{align*}
    This expression converges to $2s$ as $N\to\infty$.\\
    Next we calculate the variance. Let $1\leq i,j,k,l\leq N$ be distinct. We use the same abbreviations $P_{X_i}$ and $P_{X_i,X_j}$ as in the proof of Theorem~\ref{thm:rand_walk_ppc}. Our goal is to find an upper bound for
    \[
    \var(R(s,N))=\frac{1}{N^2}\sum_{\substack{1\leq i,j\leq N\\i\neq j}}\sum_{\substack{1\leq k,l\leq N\\k\neq l}}\left(P_{X_i-X_j,X_k-X_l}-P_{X_i-X_j}P_{X_k-X_l}\right).
    \]
    In order to achieve this, we will split this sum into different subsums, dependent on the relations between $i$, $j$, $k$ and $l$.\\
    If $j,k,l\notin\mathcal{M}(i)$, then $X_i-X_j$ and $X_k-X_l$ are independent and therefore
    \[
    P_{X_i-X_j,X_k-X_l}=P_{X_i-X_j}P_{X_k-X_l}.
    \]
    This holds true even if $j, k$ and $l$ are in the same batch for the following reason: as $X_i$ is uniformly distributed and independent of $X_j$, also $X_i-X_j$ is uniformly distributed on the torus and independent of $X_j$ (due to the translation invariance of the uniform distribution on the torus). As $X_i$ is independent of $X_k$ and $X_l$, also $X_i-X_j$ and $X_k-X_l$ are independent. Due to symmetry the same is true for $i,k,l\notin\mathcal{M}(j)$, $i,j,l\notin\mathcal{M}(k)$, $i,j,k\notin\mathcal{M}(l)$.\\
    For the remaining cases we employ the estimate
    \begin{equation}\label{eq:ineq_double_prob_2s_N}
        P_{X_i-X_j,X_k-X_l}-P_{X_i-X_j}P_{X_k-X_l}\leq P_{X_i-X_j}(1-P_{X_k-X_l})\leq \frac{2s}{N},
    \end{equation}
    which follows from \eqref{eq:eq_Exp_Xi_Xj_sN} and \eqref{eq:ineq_Exp_Xi_Xj_sN}. If $i,j,k\in\mathcal{M}(l)$, then 
    \[
    \frac{1}{N^2}\sum_{\substack{1\leq i,j,k,l\leq N\\i\neq j, l\neq k\\\\i,j,k\in\mathcal{M}(l)}}\left(P_{X_i-X_j,X_k-X_l}-P_{X_i-X_j}P_{X_k-X_l}\right)\leq\frac{1}{N^2}NM^3\frac{2s}{N}=\mathcal{O}_N\left(\frac{1}{N^2}\right).
    \]
    The subscript $N$ in $\mathcal{O}_N(\cdot)$ once again indicates that the implicit constant is independent of $N$. Finally, we look at the case $i\in\mathcal{M}(k)$, $j\in\mathcal{M}(l)$ but $\mathcal{M}(l)\neq\mathcal{M}(k)$. By using~\eqref{eq:eq_Exp_Xi_Xj_sN} we get
    \[
       \frac{1}{N^2}\sum_{\substack{1\leq i,j,k,l\leq N\\i\neq j, l\neq k\\i\in\mathcal{M}(k),j\in\mathcal{M}(l),\mathcal{M}(l)\neq\mathcal{M}(k)}} \left(P_{X_i-X_j,X_k-X_l}-P_{X_i-X_j}P_{X_k-X_l}\right)=\mathcal{O}_N\left(\frac{1}{N}\right),
    \]
    since the sum consists of $\mathcal{O}_N(N^2)$ summands, each of them not larger than $\frac{2s}{N}$, due to \eqref{eq:ineq_double_prob_2s_N}.\\  
Thus $\var(R(s,N))=\mathcal{O}_N\left(\frac1N\right)$. Using Chebyscheff's inequality we have
\[
\Prob{\vert R(s,N)-\Exp{R(s,N)}\vert\geq \frac{1}{N^{\frac14}}} = \mathcal{O}_N \left( \frac{1}{N^{\frac{1}{2}}} \right).
\]
As in the proof of Theorem~\ref{thm:rand_walk_ppc} we can prove almost sure convergence first for a subsequence and then for the entire sequence.
\end{proof}
\begin{remark} Theorem~\ref{thm:batch_js_poiss_pair_corr} can alternatively be derived by decomposing the $M$-batch jittered sample into a deterministic and a stochastic part and using a specific approximation technique introduced in \cite{lachman:additive_energy_discrepancy_poissonian}. In an earlier version of this paper, all the details were elaborated, but we think that the proof which we included here is more compact and easier to understand.
\end{remark}
    
In contrast to $M$-batch jittered sampling, sequential jittered sampling does not have PPC: for $s=\frac{1}{2}$, any $x\in [0,1)$ and $N=2^n$, the interval $\left[x-\frac{s}{N},x+\frac{s}{N}\right]$ can contain at most $2$ points of the finite sequence $X_1, \ldots, X_N$. The details of the argument are explained in the proof of the following proposition.
\begin{proposition}\label{prop:seq_js_not_PPC}
    A sequential jittered sampling sequence $\mathseq{X_n}{n\in\N}$ does not have PPC almost surely.
\end{proposition}
\begin{proof}
    For $s=\frac12$ and $n\geq 2$ we consider  the finite sequence $X_1, \ldots, X_{2^n}$. To keep the presentation compact we write $N=2^n$ and denote $X_1$ also as $X_{n+1}$ (nevertheless, we still only consider $N$ points). Since all slots $\left[\frac{k-1}{N},\frac{k}{N}\right)$ contain exactly one point, we assume without loss of generality that $X_k\in\left[\frac{k-1}{N},\frac{k}{N}\right)$. Then for every $1\leq k\leq N$, using that $X_k$ has 
    the probability density $N\mathds{1}_{\left[\frac{k-1}{N},\frac{k}{N}\right)}$, we have
\begin{align*}
    \Prob{\norm{X_{k+1}-X_k}\leq \frac{s}{N}}&=\Prob{\norm{X_{k+1}-X_k}\leq \frac{1}{2N}}\\
    &=\int_{\frac{k}{N}-\frac{1}{2N}}^{\frac{k}{N}}N\int_{\frac{k}{N}}^{x_{k}+\frac{1}{2N}}N \mathrm{d} x_{k+1} \mathrm{d} x_k=\frac18.
\end{align*}
Since the probability of two points in non-adjacent intervals having distance less than $\frac{1}{2N}$ is zero, we get 
\begin{align*}
\Exp{\frac{1}{N}\sum_{\substack{1\leq i,j\leq N\\i\neq j}}\ind_{\norm{X_i-X_j}\leq \frac{s}{N}}}&=\frac{1}{N}\sum_{\substack{1\leq i,j\leq N\\i\neq j}}\Prob{\norm{X_{i}-X_j}\leq \frac{s}{N}}\\
    &=\frac{1}{N}\sum_{k=1}^N 2\Prob{\norm{X_{k+1}-X_k}\leq \frac{s}{N}}\\
    &=\frac14=\frac{s}{2}.
\end{align*}
Since 
\[
\lim_{n\to\infty}\Exp{R\left(\frac{1}{2},2^n\right)}=\frac{1}{4},
\]
we cannot have $\lim_{N\to\infty}R(s,N)=2s$ almost surely, because $R\left(\frac12,N\right)\leq 2$ a.e. gives us uniform integrability, which means that a.s. convergence implies convergence in the mean.
\end{proof}
However, it turns out that jittered sampling sequences have weak Poissonian pair correlations. This is the content of Theorem \ref{thm:seq_LHC_weak_PPC} and the rest of this section is dedicated to its proof. Since the proof is rather technical, it is useful to introduce some notation. 
For $x,d\in[0,1)$ and $n \in \mathbb{N}$ we define
    \[
    \BOX(x,n,d):=\mathset{0\leq j < 2^n}{\frac{j}{2^n}\in(x-d,x+d)},
    \]
where as usual the interval is seen modulo $1$. If both $i, i+1\in\BOX(x,n,d)$, we have
$\left[\frac{i}{2^n},\frac{i+1}{2^n}\right)\subseteq (x-d,x+d)$, meaning that these subintervals are fully contained in $(x-d,x+d)$. All these subintervals have length $\frac{1}{2^n}$.\\[12pt]
Moreover, we will write $N \in \mathbb{N}$ as $N=2^n+d$ with $n,d \in \mathbb{N}$. This way of rewriting $N$ is useful for our proofs because of the following reason: By construction of the sequential jittered sampling we know that for every $0\leq j\leq 2^n-1$, there exists exactly one point $X_{i_0}$ with $i_0\leq 2^n$ such that $X_{i_0}\in\left[\frac{j}{2^n},\frac{j+1}{2^n}\right)$. Furthermore, the probability that there is a point $X_{i_1}\in\left[\frac{j}{2^n},\frac{j+1}{2^n}\right)$ with $2^n<i_1 \le N$ is $\frac{k}{2^n}$. This allows us to structure the proofs by looking at these two cases separately, which simplifies the problem.\\[12pt]
Finally, we introduce the shorthand notation
\[
    F(i,\alpha,s,N):=\card{\mathset{1\leq j\leq N}{ i\neq j,\norm{X_i-X_j}\leq \frac{s}{N^\alpha}}}.
\]
Note that 
\[
\Exp{R_\alpha(s,N)} = \frac{1}{N^{2-\alpha}} \sum^N_{i=1} \Exp{F(i,\alpha,s,N)}.
\]
In summary, this allows us to formulate the following rather compact point counting lemma which is a key ingredient in the proof of Theorem~\ref{thm:seq_LHC_weak_PPC}.
\begin{lemma}\label{lem:complexity_N_of_js_sums}Let $\mathseq{X_n}{n\in\N}$ be a sequential jittered sampling sequence in $[0,1)$ and $\alpha\in(0,1)$. Let $N\in\N$ and write $N=2^n+d$ with $n \in \mathbb{N}$ and $0\leq d<2^n$. Then:
\begin{enumerate}
    \item[(i)] For every $x\in [0,1)$, we have $\frac{N^\alpha}{2^n} \card{\BOX\left(x,n,\frac{s}{N^\alpha}\right)}=2s+\mathcal{O}_{N}\left(\frac{1}{N^{1-\alpha}}\right)$.
    \item[(ii)] For every $1\leq i\leq N$ the equality
    \[
    \Exp{F(i,\alpha,s,N)}=\card{\BOX\left(X_i,n,\frac{s}{N^\alpha}\right)}+\mathcal{O}_N(1)+d\frac{2s}{N^\alpha}
    \]
    holds.
    \item[(iii)] For every $1\leq i\leq N$ the  expected value satisfies
    \[
    \Exp{F(i,\alpha,s,N)}=2s N^{1-\alpha}+\mathcal{O}_N\left(1\right),
    \]
    implying
    \[
    \Exp{R_\alpha(s,N)} = 2s + \mathcal{O}_N \left( \frac{1}{N^{1-\alpha}}\right).
    \]
    \item[(iv)] It holds that
    \[
    \Exp{\sum_{\substack{1\leq i,j,k\leq N\\\text{distinct}}}\ind_{\left\{\norm{X_i-X_j}\leq \frac{s}{N^\alpha}\right\}}\ind_{\left\{\norm{X_i-X_k}\leq \frac{s}{N^\alpha}\right\}}}=\mathcal{O}_N\left(N^{3-\alpha}\right).
    \]
    \item[(v)] The inequality
    \begin{align*}
        &\Exp{\sum_{\substack{1\leq i,j,k,l\leq N\\\text{distinct}}}\ind_{\left\{\norm{X_i-X_j}\leq \frac{s}{N^\alpha}\right\}}\ind_{\left\{\norm{X_k-X_l}\leq \frac{s}{N^\alpha}\right\}}}\\
        &= \left( \Exp{\sum_{i=1}^N\sum_{\substack{j=1\\i\neq j}}^{N}\ind_{\left\{\norm{X_i-X_j}\leq \frac{s}{N^\alpha}\right\}}}\right)^2+\mathcal{O}_N(N^{3-\alpha})
    \end{align*}
    holds.
    \end{enumerate}
\end{lemma}
\begin{proof}

    (i) First we notice that $c2^n=N$ for some $c\in[1,2)$ and therefore $\frac{N^\alpha}{2^n}=\mathcal{O}\left(\frac{1}{N^{1-\alpha}}\right)$. Hence we have
    \begin{equation}\label{eq:approx_2s_2n_Nalpha}
    \left\vert\frac{\left\lceil 2s\frac{2^n}{N^\alpha}\right\rceil}{\frac{2^n}{N^\alpha}}-2s\right\vert=\frac{\left\vert\left\lceil 2s\frac{2^n}{N^\alpha}\right\rceil-2s\frac{2^n}{N^\alpha}\right\vert}{\frac{2^n}{N^\alpha}}\leq\frac{1}{\frac{2^n}{N^\alpha}}=\mathcal{O}\left(\frac{1}{N^{1-\alpha}}\right).
    \end{equation}
    The same is also true if $\left\lceil 2s\frac{2^n}{N^\alpha}\right\rceil$ is replaced by $\left\lceil 2s\frac{2^n}{N^\alpha}\right\rceil - 1$. We now assume that $N$ is large enough such that $\frac{s}{N^\alpha}<\frac12$. Then
    \[
    \left\lceil 2s\frac{2^n}{N^\alpha}\right\rceil - 1 \leq \card{\BOX\left(X_i,n,\frac{s}{N^\alpha}\right)} \leq \left\lceil 2s\frac{2^n}{N^\alpha}\right\rceil
    \]
    which in combination with \eqref{eq:approx_2s_2n_Nalpha} establishes (i).\\[12pt]
    (ii) Notice that by construction we surely know that for every $j\in\BOX\left(x_i,n,\frac{s}{N^\alpha}\right)$ there is one point in $\left(x_i-\frac{s}{N^\alpha},x_i+\frac{s}{N^\alpha}\right)$. Since the $d$ points are randomly distributed among the intervals, we expect there to be an additional $d\frac{2s}{N^\alpha}$ points within the interval. Near the end points of $\left(x_i-\frac{s}{N^\alpha},x_i+\frac{s}{N^\alpha}\right)$ there can be up to 4 additional points included, and up to 1 point may be missing in the slot where $x_i$ is located. This gives us
    \[
    \Exp{F(i,\alpha,s,N)}=\card{\BOX\left(X_i,n,\frac{s}{N^\alpha}\right)}+\mathcal{O}_N(1)+d\frac{2s}{N^\alpha}
    \]
    (iii) Follows directly from (i) and (ii) by using $1+\frac{d}{2^n}=\frac{N}{2^n}$.\\[12pt]
    (iv) From (iii) it follows that
    \begin{align*}
    &\Exp{\sum_{\substack{1\leq i,j,k\leq N\\\text{distinct}}}\ind_{\left\{\norm{X_i-X_j}\leq \frac{s}{N^\alpha}\right\}}\ind_{\left\{\norm{X_i-X_k}\leq \frac{s}{N^\alpha}\right\}}}\\
    &\leq\Exp{\sum_{k=1}^N\sum_{\substack{1\leq i,j\leq N\\i\neq j}}\ind_{\left\{\norm{X_i-X_j}\leq \frac{s}{N^\alpha}\right\}}\ind_{\left\{\norm{X_i-X_k}\leq \frac{s}{N^\alpha}\right\}}}\\
    &\leq\Exp{\sum_{k=1}^N\sum_{\substack{1\leq i,j\leq N\\i\neq j}}\ind_{\left\{\norm{X_i-X_j}\leq \frac{s}{N^\alpha}\right\}}}\\
    &=N\sum_{i=1}^N\Exp{F(i,\alpha,s,N))}=\mathcal{O}_N\left(N^{3-\alpha}\right).
    \end{align*}
    (v) We fix $i,k\leq N$ and write $B_k$ as short hand notation for $\card{\BOX\left(X_k,n,\frac{s}{N^\alpha}\right)}$. If $\norm{X_i-X_j} \leq \frac{s}{N^\alpha}$ we claim that similarly as in the proof of (ii) we have
    \begin{align} \label{eq:formula_intersection} \sum_{\substack{l=1\\l \neq k}}^N\ProbCon{\norm{X_k-X_l}\leq \frac{s}{N^\alpha}}{\norm{X_i-X_j}\leq \frac{s}{N^\alpha}}=B_k\left(1+\frac{d}{2^n}\right) + \varepsilon
    \end{align}
    for some $|\varepsilon| \leq 7$. Note that the left hand side is equal to the conditional expectation $\ExpCon{F(k,\alpha,s,N)}{ \norm{X_i-X_j}\leq \frac{s}{N^\alpha}}$. Moreover, we define the interval $I_k := [\{X_k-\frac{s}{N^\alpha}\},\{X_k+\frac{s}{N^\alpha}\}]$ around $X_k$. There are different cases for the indices $i,j,k$ which need to be addressed separately.\\[12pt] 
    The first case, which we look at, is $k \leq 2^n$ and $i,j > 2^n$. Now we consider the first $2^n$ points apart from $X_k$ and check which of them lie in $I_k$: there are $B_k-1$ elementary intervals of length $\frac{1}{2^n}$ fully contained in $I_k$. Hence, we have in any case $B_k-2$ points from $S_1:=\{X_1,\ldots,X_{2^n}\} \setminus \{ X_k \}$ in $I_k$. Furthermore, there are two additional elementary intervals with non-trivial intersection with $I_k$. These may yield two additional points. This sums up to $B_k + \varepsilon_0$ points with $|\varepsilon_0|<2$  contained in $I_k$.\\
    Next we come to the \textit{expected} number of points from $S_2:= \left\{ X_{2^n+1}, \ldots, X_{2^n+d} \right\} \setminus \{X_i, X_j\}$ in $I_k$. There are $d-2$ points in $S_2$ that can fall into a slot of length $\frac{1}{2^n}$ which intersects $I_k$ but does not contain $X_i$ and $X_j$. Furthermore, there are between $(B_k-1)-2 = B_k-3$ and $(B_k-1)+2=B_k+1$ of such slots. Therefore, the expected number of points from $S_2$ falling into $I_k$ is $(B_k+\varepsilon_1)\frac{d-2}{2^n}$ with $|\varepsilon_1| \leq 3$. The points $X_i,X_j$ may add an additional $\varepsilon_2$ with $|\varepsilon_2| \leq 2$.\\
    In total, we obtain for the expected number of points in the first case the expression
    \[
    B_k\left(1+\frac{d-2}{2^n} \right) \pm 5 = B_k\left(1+\frac{d}{2^n}\right) + \varepsilon
    \]
    with $|\varepsilon|\leq 7.$\\[12pt]
    Second we consider the case $i,k \leq 2^n, j > 2^n$. By the same argument as in the first case, there are $B_k + \varepsilon_1$ points with $|\varepsilon_1| \leq 2$ from $S_1$ in $I_k$. The set $S_2$ needs to be replaced by $\widetilde{S}_2:= \left\{ X_{2^n+1}, \ldots, X_{2^n+d} \right\} \setminus \{X_j\}$. The number of slots of length $\frac{1}{2^n}$, which intersect $I_k$ but do not contain $X_j$ is between $(B_k-1)-1 = B_k-2$ and $(B_k-1)+2 = B_k+1$. Then formula \eqref{eq:formula_intersection} can be derived by the same argument as in the first case but the error term satisfies even $|\varepsilon| \leq 5$.\\[12pt]
    All other possible cases for the indices $i,j,k$ can be treated similarly.\\[12pt]
    Thus, it follows that 
      \begin{align*}
        &\Exp{\sum_{\substack{1\leq i,j,k,l\leq N\\\text{distinct}}}\ind_{\left\{\norm{X_i-X_j}\leq \frac{s}{N^\alpha}\right\}}\ind_{\left\{\norm{X_k-X_l}\leq \frac{s}{N^\alpha}\right\}}}\\
        &\leq \sum_{1\leq i \neq j \leq N} \left( \sum_{1\leq k \neq l \leq N} \ProbCon{\norm{X_k-X_l}\leq \frac{s}{N^\alpha}}{\norm{X_i-X_j}\leq \frac{s}{N^\alpha}}\right)  \Prob{\norm{X_i-X_j} \leq \frac{s}{N^\alpha}} \\ 
        & \leq \sum_{1\leq i \neq j \leq N} \sum_{k=1}^N \left( B_{k}\left( 1 + \frac{d}{2^n}\right) + \varepsilon\right)\Prob{\norm{X_i-X_j} \leq \frac{s}{N^\alpha}},
    \end{align*}
    where in the last step we used \eqref{eq:formula_intersection}. By (iii) the last line is equal to 
    \[
      \sum_{k=1}^N \left( B_{k} \frac{N}{2^n} + \varepsilon\right) (2sN^{2-\alpha} + \mathcal{O}_N(N)).
    \]
    Using (i) we obtain
    \[
     N \left( 2sN^{1-\alpha} + \mathcal{O}_N(1)+ \varepsilon\right) (2sN^{2-\alpha} + \mathcal{O}_N(N)) = 4s^2 N^{4-2\alpha} + \mathcal{O}_N \left( N^{3-\alpha} \right).
    \]
    Finally, we note that by (iii) the right hand side in (v) equals 
    \[
    N^2 \left(2sN^{1-\alpha} + O_N(1)\right)^2 = 4s^2 N^{4-2\alpha} + \mathcal{O}_N(N^{3-\alpha}).
    \]
    Thus the claim follows.
\end{proof}
Having established these rather technical relations, the actual proof of Theorem~\ref{thm:seq_LHC_weak_PPC} becomes much more compact.
\begin{proof}[Proof of Theorem \ref{thm:seq_LHC_weak_PPC}]
The proof for (i) has already been done in Proposition \ref{prop:seq_js_not_PPC}. Therefore, we only need to prove (ii). The idea of the proof is again to calculate the limit of the expected value of $R_\alpha(s,N)$, to bound its variance and to apply Chebyscheff's inequality afterwards. Using Lemma \ref{lem:complexity_N_of_js_sums} 
(iii), we get
    \[
\lim_{N\to\infty}\Exp{R_\alpha(s,N)} = 2s.
    \]
    Next, we will estimate the variance in a similarly as in Theorem \ref{thm:batch_js_poiss_pair_corr}. We have
    \[
    R_\alpha(s,N)^2=\frac{1}{N^{4-2\alpha}}\sum_{\substack{1\leq i,j\leq N\\i\neq j}}\sum_{\substack{1\leq k,l\leq N\\ k\neq l}}\ind_{\{\norm{X_i-X_j}\leq \frac{s}{N^\alpha}\}}\ind_{\{\norm{X_k-X_l}\leq \frac{s}{N^\alpha}\}}
    \]
    The sum is now split into the cases where the set $K=\{i,j,k,l\}$ has cardinality $2$,$3$ and $4$.\\[12pt]
    \underline{Case $\card{K}=2$:}
    Due to Lemma \ref{lem:complexity_N_of_js_sums}(iii), we have
    \begin{align*}
&\frac{1}{N^{4-2\alpha}}\sum_{\substack{1\leq i,j,k,l\leq N\\i\neq j,k\neq l\\\{i,j\}=\{k,l\}}}\Exp{\ind_{\left\{\norm{X_i-X_j}\leq \frac{s}{N^\alpha}\right\}}\ind_{\left\{\norm{X_k-X_l}\leq \frac{s}{N^\alpha}\right\}}}\\
&=\frac{1}{N^{4-2\alpha}}\sum_{\substack{1\leq i,j\leq N\\i\neq j}}2\Exp{\ind_{\left\{\norm{X_i-X_j}\leq \frac{s}{N^\alpha}\right\}}}={\frac{2}{N^{2-\alpha}}\Exp{R_\alpha(s,N)}=\mathcal{O}_N\left(\frac{1}{N^{2-\alpha}}\right)}.
\end{align*}
\underline{Case $\card{K}=3$:} In this case Lemma \ref{lem:complexity_N_of_js_sums} (iv) implies
 \begin{align*}
&\frac{1}{N^{4-2\alpha}}\Exp{\sum_{\substack{1\leq i,j,k,l\leq N\\i\neq j,k\neq l\\\card{\{i,j,k,l\}}=3}}\ind_{\left\{\norm{X_i-X_j}\leq \frac{s}{N^\alpha}\right\}}\ind_{\left\{\norm{X_k-X_l}\leq \frac{s}{N^\alpha}\right\}}}\\
&=\frac{4}{N^{4-2\alpha}}\Exp{\sum_{\substack{1\leq i,j,k\leq N\\\text{distinct}}}\ind_{\left\{\norm{X_i-X_j}\leq \frac{s}{N^\alpha}\right\}}\ind_{\left\{\norm{X_i-X_k}\leq \frac{s}{N^\alpha}\right\}}}\\
&=\frac{4}{N^{4-2\alpha}}\mathcal{O}_N\left(N^{3-\alpha}\right)=\mathcal{O}_N\left(\frac{1}{N^{1-\alpha}}\right).
\end{align*}
\underline{Case $\card{K}=4$:} When all four indices are distinct, Lemma \ref{lem:complexity_N_of_js_sums} (v) gives us
 \begin{align*}
&\frac{1}{N^{4-2\alpha}}\Exp{\sum_{\substack{1\leq i,j,k,l\leq N\\\text{distinct}}}\ind_{\left\{\norm{X_i-X_j}\leq \frac{s}{N^\alpha}\right\}}\ind_{\left\{\norm{X_k-X_l}\leq \frac{s}{N^\alpha}\right\}}}\\
&=\frac{1}{N^{4-2\alpha}}\left(\Exp{\sum_{i=1}^N\sum_{\substack{j=1\\i\neq j}}^{N}\ind_{\left\{\norm{X_i-X_j}\leq \frac{s}{N^\alpha}\right\}}}^{2}
+\mathcal{O}_N(N^{3-\alpha})\right)\\
&=\Exp{R_\alpha(s,N)}^2+\mathcal{O}_N\left(\frac{1}{N^{1-\alpha}}\right).
\end{align*}
This implies that $\var(R(s,N))=\Exp{R_\alpha(s,N)^2}-\Exp{R_\alpha(s,N)}^2=\mathcal{O}_N\left(\frac{1}{N^{1-\alpha}}\right)$ and by once again using Chebyscheff's inequality we have
\[
\Prob{\vert R(s,N)-\Exp{R(s,N)}\vert\geq \frac{1}{N^{\frac{1-\alpha}{4}}}} = \mathcal{O}_N \left(\frac{1}{N^{\frac{1-\alpha}{2}}} \right).
\]
By the same argument as in the proof of Theorem \ref{thm:rand_walk_ppc}, it follows that
\[
\lim_{N\to\infty} R_\alpha(s,N)=2s~\text{almost surely}.
\]
\end{proof}

\newcommand{\etalchar}[1]{$^{#1}$}

\end{document}